\documentclass[12pt, a4paper, reqno]{amsart}
\usepackage{amsmath}
\usepackage{amsfonts}
\usepackage{amssymb}
\usepackage{amsthm}
\usepackage{url}
\usepackage{graphicx}

\newcommand{\R}{\mathbb{R}}

\DeclareMathOperator*{\dist}{dist}

\def\vint_#1{\mathchoice%
          {\mathop{\kern 0.2em\vrule width 0.6em height 0.69678ex depth -0.58065ex
                  \kern -0.8em \intop}\nolimits_{\kern -0.4em#1}}%
          {\mathop{\kern 0.1em\vrule width 0.5em height 0.69678ex depth -0.60387ex
                  \kern -0.6em \intop}\nolimits_{#1}}%
          {\mathop{\kern 0.1em\vrule width 0.5em height 0.69678ex depth -0.60387ex
                  \kern -0.6em \intop}\nolimits_{#1}}%
          {\mathop{\kern 0.1em\vrule width 0.5em height 0.69678ex depth -0.60387ex
                  \kern -0.6em \intop}\nolimits_{#1}}}

\newcommand{\art}[6]{{\sc #1, \rm #2, \it #3 \bf #4 \rm (#5), \mbox{#6}.}}
\newcommand{\book}[3]{{\sc #1, \it #2, \rm #3.}}
\newcommand{\AND}{{\rm and }}
\newcommand{\p}{{$p\mspace{1mu}$}}

\newcommand{\Om}{\Omega}
\newcommand{\loc}{_{\rm loc}}

\newcommand{\eps}{\varepsilon}

\DeclareMathOperator*{\ext}{ext}
\DeclareMathOperator*{\inte}{int}

\theoremstyle{plain}
\newtheorem{theorem}[equation]{Theorem}
\newtheorem{lemma}[equation]{Lemma}

\newtheorem{proposition}[equation]{Proposition}

\numberwithin{equation}{section}

\theoremstyle{definition}

\theoremstyle{remark}
\newtheorem{remark}[equation]{Remark}

\title[Vanishing properties of sign changing solutions]{Vanishing properties of
    sign changing solutions to \p-Laplace type equations in the plane}

\author{Seppo Granlund} \address[S.G.]{University of Helsinki,
  Department of Mathematics and Statistics, P.O. Box 68, FI-00014
  University of Helsinki, Finland} \email{seppo.granlund@pp.inet.fi}

\author{Niko Marola} \address[N.M.]{University of Helsinki, Department
  of Mathematics and Statistics, P.O. Box 68, FI-00014 University of
  Helsinki, Finland} \email{niko.marola@helsinki.fi}

\date{}

\begin{document}



\begin{abstract}
  We study the nonlinear eigenvalue problem for the \p-Laplacian, and
  more general problem constituting the Fu\v cik spectrum. We are
  interested in some vanishing properties of sign changing solutions
  to these problems.
  Our method is applicable in the plane.
\end{abstract}

\maketitle

\section{Introduction}

We consider the nonlinear eigenvalue problem
\begin{equation} \label{eq:eigpLap} 
-\nabla\cdot(|\nabla
  u|^{p-2}\nabla u)=\lambda|u|^{p-2}u, 
\end{equation}
where $1<p<\infty$, $\lambda\in\R$ is a spectral parameter, and $u=0$
on the boundary of a bounded domain $G\subset\R^2$ with smooth
boundary $\partial G$. A good introduction to the subject is
\cite{LindqvistEig} and the references given there, but see also
\cite{GarciaPeral}. In the present paper, we study some vanishing
properties of the second eigenfunction of the \p-Laplacian, i.e. we
consider a solution to \eqref{eq:eigpLap} corresponding to the second
eigenvalue $\lambda_2$. Our main result is stated in
Theorem~\ref{thm:eiguniquecont}. The method presented in the paper is
based on some topological properties of the nodal domains and the
nodal line of the second eigenfunction; our main tool is to couple the
Harnack inequality and Hopf's lemma with some topological arguments.


Due to lack of the unique continuation property, there is no analogue
of the Courant nodal domain theorem \cite{CouHil} for the nonlinear
eigenvalue problem \eqref{eq:eigpLap}. We refer to
Theorem~\ref{thm:DrabekRobinson} and the discussion in
Section~\ref{sect:Hopf}. However, it is was proved in \cite{Cuesta},
without unique continuation, that the second eigenfunction of the
\p-Laplacian has exactly two nodal domains, $\{x\in G: u(x)> 0\}$ and
$\{x\in G: u(x)<0\}$. It is a profound difficulty that very little
seems to be known about the topology and the geometry of these nodal
domains, even the case of $p=2$ in the plane is not completely known.

For $p=2$ one recovers the linear eigenvalue problem for the
Laplacian, and unique continuation is a well-known feature and the
structure of the spectrum is fully understood. As was mentioned above,
in general the structure of nodal domains is not completely known. A
conjecture of L. Payne~\cite[Conjecture 5]{Payne} states that any
second eigenfunction $u_2$ of the Laplacian on a bounded planar domain
$\Om$ does not have a closed nodal line, i.e.
\[
\overline{\{x\in\Om: u_2(x)=0\}}\cap\partial\Om
\neq\emptyset.
\]
To the best of our knowledge, it is not even known whether the
conjecture is true for bounded simply-connected planar domains. There
are, however, significant contributions. We refer to a discussion
after Proposition~\ref{prop:bN} for references.

We remark that the vanishing properties, the unique continuation
property and the geometry of the nodal line in particular, are still
an open problem for the solutions to non-linear equations, e.g. for
the solutions to the \p-Laplace equation
\begin{equation} \label{eq:pLap}
\nabla\cdot(|\nabla u|^{p-2}\nabla u)=0,
\end{equation}
although there are some results. We refer to \cite{Alessandrini87},
\cite{ALR}, \cite{AlesSiga}, \cite{BoIw}, \cite{Manfredi},
\cite{Martio}.



Our results are applicable also for solutions to a more general
nonlinear eigenvalue problem which constitutes the Fu\v cik
spectrum. Discussions and results on this problem are postponed until
Section~\ref{sect:Fucik}.

Lastly, we mention that the method in this paper is applicable, after
some modifications, to the planar solutions of the Dirichlet problem
for certain quasilinear elliptic equations
\[
\nabla\cdot \mathcal{A}(x,\nabla u)= \mathcal{B}(x,\nabla u)
\]
For these results we refer to \cite{GraMar}.

\subsection*{Notation} Throughout the paper $G$ is a bounded
simply-connected domain, a domain is an open connected set, of $\R^2$,
and in \eqref{eq:eigpLap} we have $1<p<\infty$. We use the notation $B_r =
B_r(x) = B(x,r)$ for concentric open balls of radii $r>0$ centered at
some $x\in G$.
We denote the closure, interior, exterior, and boundary of $E$ by
$\overline{E}$, $\inte(E)$, $\ext(E)$, and $\partial E$, respectively.
%

\subsection*{Acknowledgements} We thank Professor Giovanni
Alessandrini for pointing out a gap in a previous version of this
paper, and also for the example in Remark~\ref{remark}.

\section{Eigenfunctions and nodal domains}
\label{sect:Hopf}

We interpret equation \eqref{eq:eigpLap} in the weak sense. A function
$u\in W_0^{1,p}(G)$, $u$ nontrivial, is an eigenfunction if there
exists $\lambda\in \R$ such that
\begin{equation} \label{eq:weak} \int_G|\nabla u|^{p-2}\nabla u
  \cdot\nabla\eta\, dx = \lambda\int_G |u|^{p-2}u\eta\, dx,
\end{equation}
where $\eta$ is a test-function in $W_0^{1,p}(G)$. The corresponding
real number $\lambda$ is called an eigenvalue. The elliptic regularity
theory implies that $u\in C\loc^{1,\alpha}(G)$ for some $\alpha>0$,
cf. \cite{DiB, Tolksdorf}. We refer to Lindqvist~\cite{LindqvistEig}
and the references therein for an overview on nonlinear eigenfunctions
and their properties.

We mention that by approximation $u$ itself can act as a test-function
in \eqref{eq:weak} and one has that $\lambda>0$.
The least eigenvalue $\lambda_1$, called the first eigenvalue, is
attained as the infimum of the nonlinear Rayleigh quotient. The
corresponding eigenfunction is called the first eigenfunction. Let us
list some well-known features, see e.g. \cite{LindqvistEig}: The
spectrum is a closed set, $\lambda_1$ is simple, i.e., associated
first eigenfunctions are constant multiples of each other, the first
eigenfunctions are the only eigenfunctions not changing signs, we
stress that all higher eigenfunctions necessarily change their sign,
and $\lambda_1$ is isolated.

The structure of the spectrum for the eigenvalue problem for the
Laplacian is fully understood and every eigenvalue has a variational
characterization via the Rayleigh quotient. In our case of the
nonlinear eigenvalue problem there is a second eigenvalue $\lambda_2$,
that is $\lambda_2=\min_{\lambda_1<\lambda}\lambda$, which has a
variational characterization \cite{AnaneTsouli}. 

There are several methods, we refer to \cite{GarciaPeral} for one, to
obtain a sequence of variational eigenvalues,
$\{\lambda_i^\star\}_{i=1}^\infty$, such that
\[
0<\lambda_1^\star<\lambda_2^\star\leq \cdots\leq\lambda_i^\star \to
\infty
\]
as $i\to\infty$, and that $\lambda_1^\star = \lambda_1$ and
$\lambda_2^\star=\lambda_2$. It is not clear whether this sequence
gives the entire spectrum. Indeed, it remains a pertinent question how
one can exhaust the whole spectrum which, in passing, has not been
proved to be discrete.
%

Let us next turn to study nodal domains of an eigenfunction. A maximal
connected component, i.e. one not contained in any other connected
set, of the set $\{x\in G: u(x)\neq 0\}$ is called, in what follows, a
nodal domain. We denote these components by
\[
N_i^+=\{x\in G: u(x)>0\}, \quad \textrm{and} \quad N_j^-=\{x\in G:
u(x)<0\},
\]
where $i,\,j=1,2,\ldots\,$. We remark that the restriction of any
eigenfuntion to a nodal domain is the first eigenfunction with respect
to that nodal domain. It is known, in any dimension $n\geq 2$, that
any eigenfunction has only a finite number of nodal domains. We refer
to Lindqvist~\cite{LindqvistEig} for a proof.




The classical Courant nodal domain theorem, we refer to
\cite[p. 452]{CouHil}, states that any eigenfunction of the Laplacian
corresponding to the $N$-th eigenvalue has at most $N$ nodal
domains. In an interesting paper by Alessandrini~\cite{Aless}, the
validity of the Courant nodal domain theorem for eigenfunctions of
second order self-adjoint elliptic operators with the Lipschitz
continuous coefficients in the principal part was verified. He also
proves that in the plane the Courant nodal domain theorem holds also
when the coefficients are just bounded and measurable; in higher
dimensions, he proves a weakened version of the Courant theorem when
the coefficients in the principal part are H\"older
continuous. Namely, an eigenfunction corresponding to the $N$-th
eigenvalue has at most $2(N-1)$ nodal domains, Theorem 4.5 in
\cite{Aless}.

Fundamentally, proofs of the Courant nodal domain theorem are based on
the following three main tools: 
\begin{enumerate}
\item[(1)] the {\it variational characterization of eigenvalues}, i.e. the
  eigenvalues are characterized as the minimizers of the Rayleigh
  quotient over suitable sets of functions, 
\item[(2)] the {\it maximum principle},

\item[(3)] the {\it unique continuation} property.
\end{enumerate}

When one considers the nonlinear eigenvalue problem and the
corresponding Courant nodal domain theorem, several complications
arise. Apart from the problem of describing higher eigenvalues, the
unique continuation property is still an open question for the
\p-Laplacian.

We want to recall two results related to the nonlinear analogue of the
Courant nodal domain theorem. 

\begin{theorem}[Cuesta et al.~\cite{Cuesta}] \label{thm:2curvenodal}
  Suppose $u$ is an eigenfunction corresponding to the second
  eigenvalue $\lambda_2$. Then $u$ has exactly two nodal domains.
\end{theorem}

\begin{theorem}[Dr\'abek--Robinson~\cite{DrabekRobinson}] \label{thm:DrabekRobinson}
\begin{itemize}
\item[]
\item[(1)] Suppose solutions to \eqref{eq:eigpLap} satify the unique
  continuation property.  If $u$ is an eigenfunction corresponding to
  the $N$-th eigenvalue, then $u$ has at most $N$ nodal domains.
\smallskip
\item[(2)] Suppose $u$ is an eigenfunction corresponding to the $N$-th
  eigenvalue, then $u$ has at most $2(N-1)$ nodal domains.
\end{itemize}
\end{theorem}

We close this section by recalling the Harnack inequality and the
following version of Hopf's lemma. For the proof we refer to, e.g.,
\cite[Lemma A.3]{Saka} and \cite[Proposition 3.2.1]{Tolks}.

\begin{lemma}[Hopf's lemma] \label{lemma:Hopf} Let
  $\Omega\subset\R^n$, $n\geq 2$, be a bounded domain with smooth
  boundary $\partial\Om$. Let $u\in C^1(\overline{\Om})$ satisfy
  \[
  -\nabla\cdot(|\nabla u|^{p-2}\nabla u)\geq 0
  \]
  interpreted in the weak sense. Assume further that $u>0$ in $\Om$
  and
\[
u(x_0)=0 \quad \textrm{at }\ x_0\in\partial\Om.
\]
Then $\nabla u(x_0)\neq 0$.
\end{lemma}

The following theorem can be found in Trudinger~\cite{Trudinger}. The
proof is based on the Moser iteration method.

\begin{theorem}[Harnack's inequality] \label{thm:Harnack} Suppose
  $u\geq 0$ is an eigenfunction. Then the following inequality is
  valid
\[
\sup_{B_r} u\leq C\inf_{B_r} u,
\]
where $\overline{B}_{2r}\subset G$ and the constant $C$ depends on $n$
and $p$.
\end{theorem}

\section{A few facts about the plane topology}
\label{sect:topology}

We recall a few facts about the topology of planar sets; a good
reference is \cite{Newman}.

We first recall some basic concepts. Let $\Omega$ be any domain in
$\R^2$. A Jordan arc is a point set which is homeomorphic with
$[0,1]$, wheras a Jordan curve is a point set which is homeomorphic
with a circle. By Jordan's curve theorem a Jordan curve in $\R^2$ has
two complementary domains, and the curve is the boundary of each
component. One of these two domains is bounded and this domain is
called the interior of the Jordan curve. A domain whose boundary is a
Jordan curve is called a Jordan domain. 

As a related note, it is well known that the boundary of a bounded
simply-connected domain in the plane is connected. In the plane a
simply-connected domain $\Om$ can be defined by the property that all
points in the interior of any Jordan curve, which consists of points
of $\Om$, are also points of $\Om$ \cite{Nehari}.




A Jordan arc with one end-point on $\partial\Omega$ and all its other
points in $\Omega$, is called an end-cut. If both end-points are in
$\partial \Omega$, and the rest in $\Omega$, a Jordan arc is said to
be a cross-cut in $\Omega$. A point $x\in\partial\Omega$ is said to be
accessible from $\Omega$ if it is an end-point of an end-cut in
$\Omega$. Accessible boundary points of a planar domain are aplenty as
the following lemma states.

\begin{lemma}[p. 162, \cite{Newman}] \label{lemma:Newman2} Let
  $\Omega$ be any domain in $\R^2$. The accessible points of $\partial
  \Omega$ are dense in $\partial\Omega$.
\end{lemma}

\begin{lemma}[p. 118, \cite{Newman}] \label{lemma:cross-cut} If both
  end-points of a cross-cut $\gamma$ in a domain $\Om$ are on the same
  component of $\partial\Om$, then $\Om\setminus\gamma$ has two
  components, and $\gamma$ is contained in the boundaries of both.
\end{lemma}

We shall make use of the preceding lemmas, as well as lemmas below, in
the proof of our main theorem, Theorem~\ref{thm:eiguniquecont}.

In what follows, $u$ is the second eigenfunction of the
\p-Laplacian. By Theorem~\ref{thm:2curvenodal} $u$ has exactly two
nodal domains, which we denote by $N^+$ and $N^-$.

\begin{remark} \label{rmk:accessible} Let us define the set
\[
\partial N_A^+ = \{x\in\partial N^+\cap G: x \textrm{ is accessible from }
N^+\},
\]
and correspondingly $\partial N_A^-$. Let $x\in\partial N^+\cap G$ and
consider a spherical neighborhood $\overline{B}_\delta(x)\subset
G$. It is possible to select points $x_0\in B_{\delta/2}(x)\cap N^+$
and $x_\delta\in B_{\delta/2}(x)\cap\partial N^+$ such that $x_\delta$
is the closest point to $x_0$. In fact, the line segment $[x_0,\,
x_\delta]$ is contained in $\overline{N^+}$ and so $x_\delta$ is
accessible. In addition, $x_\delta\in\partial B_\rho(x_0)$, where
$\rho=|x_0-x_\delta|$. Then by Hopf's lemma, Lemma~\ref{lemma:Hopf},
$\nabla u(x_0)\neq 0$.

Since the preceding procedure can be carried out at any arbitrary
small scale $\delta>0$, we obtain that the set $\{x\in\partial N^+:
x\in\partial N_A^+,\, u(x)=0,\, \nabla u(x)\neq 0\}$ is dense in the
relative topology. The case of $\partial N^-$ is treated similarly.

We also remark that neither $N^+$ nor $N^-$ cannot have isolated
boundary points, this can be seen by applying Harnack's inequality.
\end{remark}

We then recall a few facts about connected sets and
$\eps$-chains. If $x$ and $y$ are distinct points, then an
$\eps$-chain of points joining $x$ and $y$ is a finite sequence of
points
\[
x=a_1,\, a_2,\,\ldots,\,a_k=y
\]
 such that $|a_i-a_{i+1}|\leq\eps$, for $i=1,\,\ldots,\,k-1$. A set of points is
$\eps$-connected if every pair of points in it can be joined by an
$\eps$-chain of points in the set.

\begin{lemma}[Theorem~5.1, p. 81, \cite{Newman}] \label{lemma:chain} A
  compact set $F$ in $\R^2$ is connected if and only if it is
  $\eps$-connected for every $\eps>0$.
\end{lemma}

\begin{lemma}[Theorem~1.3, p. 73, \cite{Newman}] \label{lemma:Newman3}
  If a connected set of points in $\R^2$ intersects  both $\Omega$
  and $\R^2\setminus\Omega$ it intersects  $\partial \Omega$.
\end{lemma}

We shall need in the proof of Theorem~\ref{thm:eiguniquecont} the
observation that either $\partial N^+$ or $\partial N^-$ is
necessarily a continuum, i.e. a compact connected set with at least
two points. To show this we shall require that $G$ is a bounded
simply-connected domain.

\begin{proposition} \label{prop:bN} Suppose that $G$ is a bounded
  simply-connected domain. Then, at least, either $\partial N^+$ or
  $\partial N^-$ is a continuum.
\end{proposition}

\begin{proof}
  If either $N^+$ or $N^-$ is simply-connected, then the corresponding
  boundary is a continuum. We consider the nodal domain $N^-$ (the
  reasoning for $N^+$ is symmetric) and shall conclude that either
  $\partial N^+$ or $\partial N^-$ is a continuum. 

  We suppose, therefore, that $N^-$ is not simply-connected. Then
  there exists a Jordan curve $\gamma\subset N^-$ with its interior
  $S_\gamma$. Moreover, $S_\gamma\subset G$ since $G$ is
  simply-connected, and the set $E=\{x\in S_\gamma: u(x)\geq 0\}$ is
  non-empty. If $\tilde E=\{x\in S_\gamma: u(x)>0\}$ was empty, then
  $u(x)=0$ for all $x\in E$, and $u(x)<0$ for all $x\in
  S_\gamma\setminus E$. This is impossible by Harnack's inequality,
  Theorem~\ref{thm:Harnack}. Hence $N^-$ is simply-connected and
  $\partial N^-$ a continuum.

  We consider the case in which $\tilde E=\{x\in S_\gamma:
  u(x)>0\}\subset E$ is non-empty. The set $\tilde E$ is open and
  $\tilde E\subset N^+$. Since by Theorem~\ref{thm:2curvenodal} there
  exist exactly two nodal domains we must have that $\tilde E =
  N^+$. It follows also from Theorem~\ref{thm:2curvenodal} that $N^+$
  must be simply-connected as otherwise, by repeating the preceding
  reasoning, we would obtain a third nodal domain $\tilde N^-$. Hence
  it follows that the boundary $\partial N^+$ is a continuum.
\end{proof}

The topology of the nodal domains is not known in general. We recall
here a conjecture due to Payne~\cite[Conjecture 5]{Payne} which states
that any second eigenfunction $u_2$ of the Laplacian on a bounded
planar domain $\Om$ does not have a closed nodal line, i.e.
\[
\overline{\{x\in\Om: u_2(x)=0\}}\cap\partial\Om
\neq\emptyset.
\]
In this case, the nodal line intersects $\partial \Om$ in exactly two
points. See also Yau~\cite{Yau}.

Significant results have been obtained. To name a few,
Jerison~\cite{Jerison} proved the conjecture for long thin convex sets
without any assumption on the smoothness of the sets,
Melas~\cite{Melas} for convex domains with $C^\infty$-boundary, and
Alessandrini~\cite{Alessandrini} for general convex domains. See also
the references in these papers.

Hoffmann-Ostenhof et al.~\cite{Hoffmann} constructed a non-convex, not
simply-connected planar domain $\Om$ for which the nodal line of the
second eigenfunction of the Laplacian is closed,
i.e. 
\[
\overline{\{x\in\Om: u_2(x)=0\}}\cap\partial\Om = \emptyset. 
\]
To the best of our knowledge, it is not known, even in this linear
case, whether Payne's conjecture holds for bounded simply-connected
planar domains (see Remark 3 in
\cite{Hoffmann}). Fournais~\cite{Fournais} constructed a set in
$\R^n$, $n\geq 2$, for which the nodal surface of the second
eigenfunction of the Laplacian is closed, and thus generalizing the
domain in \cite{Hoffmann} to higher dimensions by an alternative
argument.

We also note that a recent paper \cite{Kennedy} contains an example
of a multiply connected domain in $\R^2$ for which the second
eigenfunction of the Laplacian with Robin boundary conditions has an
interior nodal line.

\section{Some vanishing properties of the second eigenfunction}
\label{sect:unique}

The following is our main theorem.

\begin{theorem} \label{thm:eiguniquecont} Suppose $u$ is an
  eigenfunction corresponding to the second eigenvalue $\lambda_2$ in
  a bounded simply-connected domain $G$ in $\R^2$. We assume further
  that for all $x\in G$ there exists $r_x>0$ such that for all $r\leq
  r_x$ the set $\{z\in B_r(x)\subset G:\ u(z)=0\}$ is connected. Then
  if $u=0$ in an open subset of $G$, then $u\equiv 0$ in $G$.
\end{theorem}

\begin{proof}
  By Theorem~\ref{thm:2curvenodal} $u$ has exactly two nodal domains
  $N^+$ and $N^-$. We assume, on the contrary, that \smallskip
  \begin{enumerate}
  \item[(A)] $u$ vanishes in a maximal open set $D\subset G$ but is
    not identically zero in $G$.
  \end{enumerate}
  The maximal open set $D$ is formed as follows: for every $x\in G$
  for which there exists an open neighborhood such that $u\equiv 0$ on
  this neighborhood we denote by $B(x,r_x)$, $r_x=\sup\left\{t>0:
    u|_{\partial B(x,t)}\equiv 0\right\}$, the maximal open
  neighborhood of $x$ where $u$ vanishes identically. Then the maximal
  open set $D$ is simply the union of all such neighborhoods. We pick
  a connected component of $D$, still denoted by $D$.

  We first show that antithesis (A) implies that any neighborhood of
  $x\in\partial D\cap G$ contains points of both nodal domains $N^+$
  and $N^-$. 

  Suppose there existed a point $x\in\partial D\cap G$ and its
  spherical neighborhood $B_\delta(x)$, $\delta>0$, such that
  $\overline{B}_\delta\subset G$ and $B_\delta(x)\cap\ext(D)$ contains
  only points of either $N^+$ or $N^-$, i.e. points at which either
  $u>0$ or $u<0$. Assume, without loss of generality, that
  $B_\delta(x)\cap\ext(D)$ contains points of $N^+$ only. Then $u\geq
  0$ on $B_\delta(x)$ and by Harnack's inequality,
  Theorem~\ref{thm:Harnack}, $u\equiv 0$ on $B_{\delta/2}(x)$, which
  contradicts the maximality of the set $D$, and hence also the
  antithesis (A). In this case our claim follows.
  
  We, therefore, assume that for any $x\in\partial D\cap G$ and for
  any $\delta>0$ the spherical neighborhood $B_\delta(x)$ contains
  points of both nodal domains $N^+$ and $N^-$, and
  $B_\delta(x)\subset G$. Hence
  \begin{equation} \label{eq:Dinclusion} (\partial D\cap G)\subset
    (\partial N^+\cap\partial N^-\cap G).
  \end{equation}

Let $\partial D_A = \{x\in\partial D: x \textrm{ is accessible from }
D\}$. By Lemma~\ref{lemma:Newman2} accessible boundary points are
dense. We select (Lemma~\ref{lemma:Newman3}) points $x_1$ and $x_2$ in
the set $\partial D_A\cap G$ and the associated spherical
neighborhoods $B_\delta(x_1)$ and $B_\delta(x_2)$ such that
$\overline{B}_\delta(x_1)\cap\overline{B}_\delta(x_2)=\emptyset$, and
that $\overline{B}_\delta(x_1),\,\overline{B}_\delta(x_2)\subset G$.

By Proposition~\ref{prop:bN} we may assume, without loss of
generality, that $\partial N^+$ is a continuum. In addition, we note
that it is easy to check that $\partial N^+\cap G = \partial N^-\cap
G$. Then we select $x_3\in B_\delta(x_1)\cap\partial N_A^+$ and
$x_4\in B_\delta(x_2)\cap\partial N_A^+$. We note that $\delta$ can be
chosen small enough so that the sets $B_\delta(x_1)\cap\partial N_A^+$
and $B_\delta(x_2)\cap\partial N_A^+$ are both connected. This is
assured by our extra assumption in the formulation of the theorem.

We connect $x_1$ to $x_2$ by a cross-cut $\gamma_D$ in $D$, and $x_3$
to $x_4$ by a cross-cut $\gamma_{N^+}$ in $N^+$. We remark that $x_3$,
and analogously $x_4$, is accessible in $N^+$ with a line segment, see
Remark~\ref{rmk:accessible}. Also $x_1$, and analogously $x_2$, is
accessible in $D$ with a line segment. We fix such line segments to
access the points $x_1,\,x_2,\,x_3$, and $x_4$. In this way the line
segments constitute part of the cross-cut $\gamma_D$ and
$\gamma_{N^+}$, respectively.

Since the boundary $\partial N^+$ is connected it is also
$\eps$-connected for every $\eps>0$. Hence for each $\eps>0$ the
points $x_1$ and $x_3$ can be joined by an $\eps$-chain
$\{a_1,\,\ldots,\,a_k\}\subset\partial N^+\cap G$, $k=k(\eps)$, such
that
\[
x_1 = a_1,\, a_2,\, \ldots, a_{k-1},\, a_k = x_3.
\]
We consider a collection of open balls
$\{B_{\frac{3}{2}\eps}(a_i)\}_{i=1}^k$, $a_i\in \partial N^+\cap G$,
such that $\overline{B}_{\frac{3}{2}\eps}(a_i)\subset G$, and a domain
$U_\eps^1$ which is defined to be
\[
U_\eps^1 = \bigcup_{i=1}^kB_{\frac{3}{2}\eps}(a_i).
\]
Since $U_\eps^1$ is a domain there exists a Jordan arc,
$\gamma_{x_1x_3}^\eps$, connecting $x_1$ to $x_3$ in
$U_\eps^1$. Correspondingly, the points $x_2$ and $x_4$ can be joined
by an $\eps$-chain in $\partial N^+$ and we obtain a domain $U_\eps^2$
and a Jordan arc $\gamma_{x_2x_4}^\eps$ connecting $x_2$ to $x_4$ in
$U_\eps^2$.

It is worth noting that we have selected $\gamma_{x_1x_3}^\eps$ and
$\gamma_{x_2x_4}^\eps$ such that either of them does not intersect
$\gamma_D$ or $\gamma_{N^+}$, save the points $x_1$ and $x_2$, and
$x_3$ and $x_4$, respectively. This is possible because of the line
segment construction descibed above.

From the preceding Jordan arcs we obtain a Jordan curve $\Gamma^\eps$,
and by slight abuse of notation we write it as a product
\[
\Gamma^\eps=\gamma_{x_1x_3}^\eps\cdot\gamma_{N^+}\cdot\gamma_{x_2x_4}^\eps\cdot\gamma_D.
\]
The Jordan curve $\Gamma^\eps$ divides the plane into two disjoint
domains, and $\Gamma^\eps$ constitutes the boundary of both
domains. We consider the bounded domain, denoted by $T_\eps$, enclosed
by $\Gamma^\eps$. See Figure~\ref{Figure}.

We next deal with the Jordan domain $T_\eps$. There exists at least
one point $y\in T_\eps$ such that $u(y)<0$, i.e. $y\in N^-$. Assume
that this is not the case: then $u(x)\geq 0$ for every $x\in
T_\eps$. Recall that $\gamma_D$ is one of the Jordan arcs which
constitutes the boundary of $T_\eps$. It hence follows that $T_\eps$
contains points of $D$ (Lemma~\ref{lemma:cross-cut}). By Harnack's
inequality, Theorem~\ref{thm:Harnack}, $u\equiv 0$ in $T_\eps$. This
is, however, impossible since $\gamma_{N^+}$ constitutes the boundary
of $T_\eps$, thus $u>0$ on a sufficiently small neighborhood of a
point in $\gamma_{N^+}$.

In an analogous way, it is possible to show that there exists a point
$z\in N^-\cap (G\setminus\overline{T}_\eps)$. We then connect $z$ and
$y$ in $N^-$ by a Jordan arc $\gamma_{zy}$. Observe that $u(x)<0$ for
every $x\in \gamma_{zy}$.

Lemma~\ref{lemma:Newman3} implies that the Jordan arc $\gamma_{zy}$ as
a connected set intersects $\Gamma^\eps$ at least at one point. We
distinguish next four possible cases for the point of intersection.

If the point of intersection is contained in $\gamma_D$ or in
$\gamma_{N^+}$ we have reached a contradiction as $u(x)=0$ for every
$x\in \gamma_D$ and $u(x)>0$ for every $x\in \gamma_{N^+}$.

Consider $\gamma_{x_1x_3}^\eps$ and $\gamma_{x_2x_4}^\eps$, and the
point of intersection which we denote by $x_\eps$ for every
$\eps>0$. We can select an appropriate subsequence
$\{x_{\eps_j}\}_{j=1}^\infty$, $\lim_{j\to\infty}\eps_j=0$, such that
for each $j$ either $x_{\eps_j}\in U_{\eps_j}^1$ or $x_{\eps_j}\in
U_{\eps_j}^2$. We assume, without loss of generality, that
$x_{\eps_j}\in U_{\eps_j}^1$. The sequence $\{x_{\eps_j}\}$ is clearly
bounded, and hence there exists a subsequence, still denoted
$\{x_{\eps_j}\}_{j=1}^\infty$, such that
\[
\lim_{j\to\infty}x_{\eps_j} = x_0,
\]
and $x_0\in\gamma_{zy}$ since $\gamma_{zy}$ is a compact set. Observe
that each $x_{\eps_j}\in B_{\frac{3}{2}\eps_j}(a_m)$ for some
$a_m\in \partial N^+\cap G$ in the $\eps_j$-chain. We note that
$u(a_m)=0$. Moreover, if there existed $\delta_0$ and a subsequence,
still denoted $\{x_{\eps_j}\}_{j=1}^\infty$, such that
\[
|u(x_{\eps_j})| \geq \delta_0 >0
\]
for every $x_{\eps_j}$, this would contradict with uniform continuity
of $u$ (note that $u$ is uniformly continuous on compact subsets of
$G$). We hence have that
\[
u(x_0)=\lim_{j\to\infty}u(x_{\eps_j})=0.
\]
In conclusion, we have reached a contradiction since $u(x_0)=0$ but,
on the other hand, $x_0\in \gamma_{zy}$ and hence $u(x_0)<0$.

All four cases lead to a contradiction. Hence antithesis (A) is
false, thus the claim follows.
\end{proof}

\begin{figure}[!htbp] 
\centering
\includegraphics[width=0.8\textwidth]{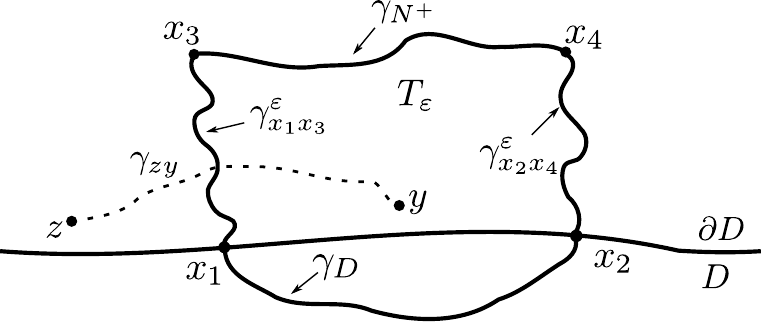}
\caption{
 Jordan domain $T_\eps$ and Jordan curve $\gamma_{zy}$ (dotted line)
 connecting $z$ to $y$ in $N^-$.}
\label{Figure}
\end{figure}

Let us discuss our extra assumption in Theorem~\ref{thm:eiguniquecont}.

\begin{remark}
  We mention that the extra assumption, for all $x\in G$ there exists
  $r_x>0$ such that for all $r\leq r_x$ the set $\{z\in
  \overline{B}_r(x):\ u(z)=0\}$ is connected, could be replaced with
  the assumption that the set has finitely many components.
\end{remark}

\begin{remark}
  The extra assumption is closely related to the concept of
  topological monotonicity or quasi-monotonicity introduced by Whyburn
  in \cite{Whyburn}; we also refer to Astala et al.~\cite[20.1.1,
  pp. 530 ff]{Astala}.

  Let us try to clarify the role of this assumption in the proof of
  the preceding theorem. We fix there the point $x_1\in \partial
  D_A\cap G$, its neighborhood $B_\delta(x_1)$, and the point $x_3\in
  B_\delta(x_1)\cap\partial N_A^+$ (similarly $x_2\in \partial D_A\cap
  G$, $B_\delta (x_2)$, and $x_4\in B_\delta(x_2)\cap\partial
  N_A^+$). At $x_1$ and $x_3$ the function $u$ is known to
  vanish. Using the extra assumption in
  Theorem~\ref{thm:eiguniquecont}, we may conclude that there indeed
  exists a continuum $\mathcal{C}_\delta$ that connects $x_1$ to $x_3$
  \emph{in} $B_\delta(x_1)$ so that $u(x)=0$ for every
  $x\in\mathcal{C}_\delta$, or in other words, that the set
  $B_\delta(x_1)\cap\partial N_A^+$ is connected.


\end{remark}

\begin{remark} \label{remark} The following example of possible
  spiral-like behavior, kindly provided us by Giovanni Alessandrini,
  illustrates the role of our extra assumption in
  Theorem~\ref{thm:eiguniquecont}.

  Let $G=B_1(0)$ and consider the function $u$ as follows
\begin{equation*}
  u(re^{i\theta})= \left\{\begin{array}{ll}
      (1-r)(2r-1)^2\sin(\theta-\log(2r-1)), & \frac1{2}<r\leq 1, \\
      0, & 0\leq r\leq \frac1{2}.\end{array}\right.
\end{equation*}
It is important to note that the function $u$ above is not known to be
a solution to \eqref{eq:eigpLap} for any $\lambda$. It is
differentiable in $G$ and its gradient vanishes on
$\overline{B}_{\frac1{2}}(0)$. This function has the following nodal
domains
\begin{align*}
  N^+ & = \left\{z=re^{i\theta}:\ 0<\theta-\log(2r-1) <\pi,\, \frac1{2}<r<1\right\}, \\
  N^- & = \left\{z=re^{i\theta}:\ \pi<\theta-\log(2r-1) <2\pi,\,
    \frac1{2}<r<1\right\}.
\end{align*}
These nodal domains are simply-connected and their boundaries contain
each one half of the circle $\partial B_1(0)$. In addition, the nodal
``line'' is the union of the closed disk $\overline{B}_{\frac1{2}}(0)$
and the two spirals
\begin{align*}
  S_1 & = \left\{z=re^{i\theta}:\ \theta=\log(2r-1),\, \frac1{2}<r<1\right\}, \\
  S_2 & = \left\{z=re^{i\theta}:\ \theta=\pi+\log(2r-1),\,
    \frac1{2}<r<1\right\}.
\end{align*}
It is not known to us if the aforementioned spiral-like scenario can
be ruled out in our method. This is why the existence of a continuum
$C_\delta$ as discussed in the preceding remark is not guaranteed
without an extra assumption, see Figure~\ref{Figure2}. It is an open
research question whether this assumption could be omitted. As the
above example shows, even if the number of nodal domains is finite in
a domain, in a small spherical neighborhood of a point there can be
infinitely many parts of these nodal domains.

\begin{figure}[!lhtbp] 
\centering
\includegraphics[width=0.9\textwidth]{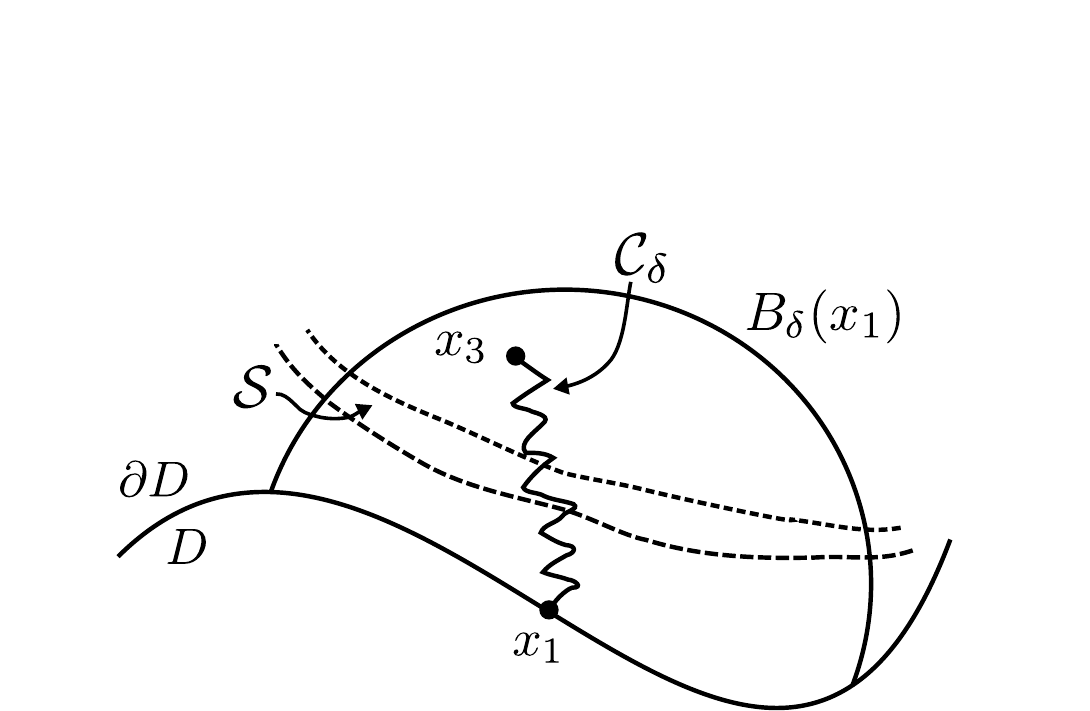}
\caption{A situation in which spiral $\mathcal{S}$, on which $u>0$ or
  $u<0$, destroys the existence of a continuum $\mathcal{C}_\delta$,
  on which $u=0$, for each $\delta>0$.}
\label{Figure2}
\end{figure}

\end{remark}

\section{Fu\v cik spectrum}
\label{sect:Fucik}

We may consider the more general equation 
\begin{equation} \label{eq:generaleig} -\nabla\cdot(|\nabla
  u|^{p-2}\nabla u)=\alpha|u|^{p-2}u_+-\beta|u|^{p-2}u_-,
\end{equation}
where $1<p<\infty$, $u_+=\max\{u,0\}$, $u_-=-\min\{u,0\}$, $\alpha,\,
\beta \in\R$ are spectral parameters, and $u=0$ on the boundary of
$G$. If a nontrivial $u\in W_0^{1,p}(G)$ satisfies
\eqref{eq:generaleig} in the weak sense it is called a Fu\v cik
eigenfunction. The corresponding pair $(\alpha,\,\beta)$ is a Fu\v cik
eigenvalue. The set of all Fu\v cik eigenvalues is the Fu\v cik
spectrum $\Sigma_p$ and, clearly, the spectrum of \eqref{eq:eigpLap}
is contained in $\Sigma_p$.

The first two curves $\mathcal{C}_1$ and $\mathcal{C}_2$ of $\Sigma_p$
can be considered as the analogue of the first two eigenvalues
$\lambda_1$ and $\lambda_2$ of the spectrum of \eqref{eq:eigpLap}. Any
pair $(\alpha,\, \beta)\in\Sigma_p$ satisfies $\alpha\geq\lambda_1$
and $\beta\geq\lambda_1$, moreover,
$\mathcal{C}_1=(\{\lambda_1\}\times\R) \cup (\R\times\{\lambda_1\})$
due to the fact that the first eigenfunctions to \eqref{eq:eigpLap} do
not change signs. Analogously, any Fu\v cik eigenfunction associated
to a Fu\v cik eigenvalue $(\alpha,\,
\beta)\notin(\{\lambda_1\}\times\R) \cup (\R\times\{\lambda_1\})$
changes sign. The second curve $\mathcal{C}_2$ was considered and
constructed in \cite{CuestaJDiffEq}, and roughly it can be defined to
be such a continuous decreasing curve in the $(\alpha,\,\beta)$-plane
that passes through $(\lambda_2,\,\lambda_2)$ and does not intersect
$\mathcal{C}_1$. Any pair $(\alpha,\,\beta)$ between $\mathcal{C}_1$
and $\mathcal{C}_2$ does not belong to $\Sigma_p$.

It is known that any Fu\v cik eigenfunction corresponding to the Fu\v
cik eigenvalue $(\alpha,\beta)\in\mathcal{C}_2$ has finite number of
nodal domains. We refer to \cite{DrabekRobinson}. We state the
counterpart of Theorem~\ref{thm:2curvenodal} in the case of Fu\v cik
spectrum as a theorem. We refer to \cite{DrabekRobinson} for an
alternative proof of the following result.

\begin{theorem}[Cuesta et al. \cite{Cuesta}] \label{thm:Fuciknodal}
  Suppose $u$ is a Fu\v cik eigenfunction corresponding to the eigenvalue
  $(\alpha,\beta)\in\mathcal{C}_2$. Then $u$ has exactly two nodal
  domains.
\end{theorem}

We consider \eqref{eq:generaleig} in a bounded simply-connected domain
$G$ in $\R^2$. Having Theorem~\ref{thm:Fuciknodal} and Harnack's
inequality \cite{Trudinger} at our disposal, we may state the
following theorem; Theorem~\ref{thm:Fuciknodal} together with the
proof of Theorem~\ref{thm:eiguniquecont} justifies the claim.

\begin{theorem}
  Suppose $u$ is a Fu\v cik eigenfunction corresponding to the
  eigenvalue $(\alpha,\,\beta)\in\mathcal{C}_2$ in $G$.  We assume
  further that for all $x\in G$ there exists $r_x>0$ such that for all
  $r\leq r_x$ the set $\{z\in B_r(x)\subset G:\ u(z)=0\}$ is connected. Then 
  if $u=0$ in an open subset of $G$, then $u\equiv 0$ in $G$.
\end{theorem}

\end{document}